\documentclass{amsart}
\usepackage{amssymb}
\usepackage{amsthm}
\newtheorem{theor}{Theorem}[section] 

\newtheorem{cons}{Consequence}[section]
\theoremstyle{definition} \newtheorem{defin}{Definition}[section]
\newtheorem{ex}{Example}[section]
\theoremstyle{remark} \newtheorem{rem}{Remark}[section]
\newcommand{\pn}{\par\noindent} \newcommand{\pmn}{\par\medskip\noindent}

\begin{document}
\title{The geometry of quadrangular convex pyramids}
\author{Yury Kochetkov}
\date{}
\begin{abstract} A convex quadrangular pyramid
$ABCDE$, where $ABCD$ is the base and $E$ --- the apex, is called
\emph{strongly flexible}, if it belongs to a continuous family of
pairwise non-congruent quadrangular pyramids that have the same
lengths of corresponding edges. $ABCDE$ is called \emph{strongly
rigid}, if such family does not exist. We prove the strong
rigidity of convex quadrangular pyramids and prove that strong
rigidity fails in the self-intersecting case. Let
$L=\{l_1,\ldots,l_8\}$ be a set of positive numbers, then a
\emph{realization} of $L$ is a convex quadrangular pyramid $ABCDE$
such, that $|AB|=l_1$, $|BC|=l_2$, $|CD|=l_3$, $|DA|=l_4$,
$|EA|=l_5$, $|EB|=l_6$, $|EC|=l_7$, $|ED|=l_8$. We prove that the
number of pairwise non-congruent realizations is $\leqslant 4$ and
give an example of a set $L$ with three pairwise non-congruent
realizations.
\end{abstract}

\email{yukochetkov@hse.ru, yuyukochetkov@gmail.com} \maketitle

\section{Introduction}
\pn A polyhedron $M$ in the three dimensional space $\mathbb{R}^3$
is called \emph{flexible} (see \cite{Con2}, \cite{Lyu}), if there
exists a continuous family of polyhedra $M_t$, $0\leqslant t$,
where
\begin{enumerate}
\item $M_0=M_t$; \item polyhedra $M_t$ have the same combinatorial
structure, as M; \item corresponding faces of $M$ and $M_t$ are
congruent; \item angles between (some) faces of $M$ and
corresponding faces of $M_t$ are different. \end{enumerate} A not
flexible polyhedron is called \emph{rigid}. The Cauchy Rigidity
Theorem states that a convex polyhedron is rigid (see \cite{Con2},
\cite{Lyu}). However, a non-convex polyhedron can be flexible
\cite{Con1}. \pmn We introduce a notion of the \emph{strong
flexibility} and the \emph{strong rigidity}.
\begin{defin}  A polyhedron $M$ in the three dimensional space $\mathbb{R}^3$
is \emph{strongly flexible}, if there exists a continuous family
of polyhedra $M_t$, $0\leqslant t$, where
\begin{enumerate} \item $M_0=M$; \item polyhedra $M_t$ have the
same combinatorial structure, as $M$; \item corresponding edges of
$M$ and $M_t$ are equal; \item some face(s) of $M$ and the
corresponding face(s) of $M_t$ are not congruent.
\end{enumerate} \end{defin} A not strongly flexible polyhedron is called
\emph{strongly rigid}.
\begin{rem} A cube is rigid, but strongly flexible. A triangular pyramid
is, of course, rigid and strongly rigid. \end{rem} \pn  A convex
quadrangular pyramid is the simplest polyhedron (after triangular
pyramid). We will prove the following statement.\pmn {\bf Theorem
3.1.} \emph{A convex quadrangular pyramid is strongly rigid.} \pmn
A non-convex quadrangular pyramid is also strongly rigid ({\bf
Consequence 3.1.}), but strong rigidity fails in the
self-intersecting case ({\bf Example 3.1.}). \pmn Our quadrangular
pyramids will be labelled, i.e. $A,B,C,D$ will be vertices of base
in order of going around it and $E$ will be the apex. For a given
set $L$ of positive numbers $L=\{l_1,\ldots,l_8\}$ we ask about
the existence of a labelled quadrangular pyramid $ABCDE$ such that
$|AB|=l_1$, $|BC|=l_2$, $|CD|=l_3$, $|DA|=l_4$, $|EA|=l_5$,
$|EB|=l_6$, $|EC|=l_7$ and $|ED|=l_8$. Such pyramid will be called
a \emph{realization} of the set $L$. \pmn {\bf Theorem 4.1.}
\emph{The number of pairwise non-congruent realizations of a set
$L$ is $\leqslant 4$.} \pmn We give an example ({\bf Example
4.1.}) of the set with three pairwise non-congruent realizations.

\section{Strong flexibility}
\pn \begin{theor} A generic polyhedron in $\mathbb{R}^3$ is strongly rigid.
\end{theor} \begin{proof} In what follows by $k$-face of a polyhedron we will
understand a face with $k$ vertices. Let the number of $k$-faces
of a polyhedron $M$ be $n_k$, $k=3,4,\ldots,m$. Then it has
$e=\frac 12 \sum_{i=3}^m i\cdot n_i$ edges and
$$v=r+2-\sum_{i=3}^m n_i=\frac{\sum_{i=3}^m (i-2)\cdot n_i}{2}+2$$ vertices.
Let us assume that some $m$-face rigidly belongs to $xy$-plane and
some edge of this face is rigidly fixed. Then vertices of this
face have $2(m-2)$ degrees of freedom and all other vertices have
$3(v-m)$ degrees of freedom. Thus, all vertices have in sum
$$2(m-2)+3(v-m)=\frac{3\cdot\sum_{i=3}^m (i-2)\cdot n_i}{2}-m+2$$ degrees
of freedom. But we have relations also:
\begin{itemize} \item lengths of all edges are fixed --- $(r-1)$
relations; \item vertices of each face are contained in one plane
--- $(i-3)$ relations for each $i$-face. \end{itemize} Thus,
the number of relations is
$$r-1+\sum_{i=3}^m n_i\cdot(i-3)-(m-3)=\frac{3\cdot\sum_{i=3}^m
(i-2)\cdot n_i}{2}-m+2.$$ We see, that the number of relations
equals the number of degrees of freedom, thus, $M$ is strongly
rigid.
\end{proof}
\begin{rem} Only polyhedra with symmetries can be strongly flexible.
\end{rem}

\section{Strong rigidity of a convex quadrangular pyramid}
\pn \begin{theor} A convex quadrangular pyramid is strongly
rigid.\end{theor} \begin{proof} We will assume that the base
$ABCD$ of a quadrangular pyramid $ABCDE$ belongs the the
$xy$-plane, vertex $A$ is at origin, vertex $B$ has coordinates
$(1,0)$, the quadrangle $ABCD$ belongs to the upper half-plane and
the apex $E$ belongs to the upper half-space. Let coordinates of
the vertex $D$ be $(a_1,b_1)$, of the vertex $C$
--- $(a_2,b_2)$ and of the vertex $E$ --- $(a_3,b_3,c_3)$. Let us
assume that $ABCDE$ is strongly flexible and there exists a
continuous deformation $A'B'C'D'E'$, where
$$A'=(0,0),\,B'=(1,0),\, C'=(a_2+x_2,b_2+y_2),\, D'=(a_1+x_1,b_1+y_1),\,
E'=(a_3+x_3,b_3+y_3,c_3+z_3)$$ and the following system holds:
$$\left\{\begin{array}{l} (a_1+x_1)^2+(b_1+y_1)^2=a_1^2+b_1^2\\
(a_2+x_2-1)^2+(b_2+y_2)^2=(a_2-1)^2+b_2^2\\ (a_2+x_2-a_1-x_1)^2+
(b_2+y_2-b_1-y_1)^2=(a_2-a_1)^2+(b_2-b_1)^2\\ (a_3+x_3)^2+(b_3+y_3)^2+
(c_3+z_3)^2=a_3^2+b_3^2+c_3^2\\ (a_3+x_3-1)^2+(b_3+y_3)^2+c_3+z_3)^2=
(a_3-1)^2+b_3^2+c_3^2\\ (a_3+x_3-a_1-x_1)^2+(b_3+y_3-b_1-y_1)^2+
(c_3+z_3)^2=(a_3-a_1)^2+(b_3-b_1)^2+c_3^2\\ (a_3+x_3-a_2-x_2)^2+
(b_3+y_3-b_2-y_2)^2+(c_3+z_3)^2=(a_3-a_2)^2+(b_3-b_2)^2+c_3^2
\end{array}\right.$$ The elimination of variables (see \cite{Cox})
$x_3,y_3,z_3,x_2,y_2$ and $y_1$ from this system gives us a
polynomial \\$R(x_1,a_1,b_1,a_2,b_2,a_3,b_3,c_3)$ of degree 3 in
variable $x_1$.\pmn Thus, we have a new system
$$\left\{\begin{array}{l} r_0(a_1,b_1,a_2,b_2,a_3,b_3,c_3)=0\\
r_1(a_1,b_1,a_2,b_2,a_3,b_3,c_3)=0\\ r_2(a_1,b_1,a_2,b_2,a_3,b_3,c_3)=0\\
r_3(a_1,b_1,a_2,b_2,a_3,b_3,c_3)=0\end{array}\right.$$ where
$r_0,r_1,r_2,r_3$ are coefficients of the polynomial $R$, as
polynomial in $x_1$. The elimination of variables
$b_1,a_3,b_3,c_3$ from this system gives us two solutions:
$$a_2=a_1+1\text{ and } a_1=\dfrac{a_2^3-a_2^2+a_2b_2^2+b_2^2}{a_2^2+b_2^2}\,
.$$ The second solution gives
$$b_1=\dfrac{b_2\cdot(a_2^2-2a_2+b_2^2)}{a_2^2+b_2^2}\Rightarrow
\begin{vmatrix} \,a_2&b_2\\ \,a_1&b_1\end{vmatrix}=-b_2<0\,.$$ Thus, we have
a clockwise rotation from the vector $\overline{OC}$ to the vector
$\overline{OD}$, i.e. the quadrangle $ABCD$ is not convex. \pmn If
$a_2=a_1+1$, then it is easy to obtain, that $b_2=b_1$, $b_3=\frac
12 b_1$ and $a_3=\frac 12\cdot(a_1+1)$, i.e. the base is a
parallelogram and the apex is just above its center $O$. Thus,
$|EA|=|EC|$ and $|EB|=|ED|$. \pmn Let $ABCDE$ be strongly flexible
and $A_1B_1C_1D_1E_1$ be a member of our family. Then
$A_1B_1C_1D_1$ is also a parallelogram with the same lengths of
edges. As $|E_1A_1|=|E_1C_1|$ and $|E_1B_1|=|E_1D_1|$, then apex
$E_1$ is just above the center $O_1$ of the base. Let
$|A_1O_1|>|AO|$, then $|E_1O_1|<|EO|$ (because $|E_1A_1|=|EA|$).
But then $|B_1O_1|<|BO|$, thus $|E_1B_1|<|EB|$. Contradiction.
\end{proof}
\begin{cons}
A non-convex quadrangular pyramid is strongly rigid.\end{cons}
\begin{proof} Using rotations, shifts and scalings we can assume,
that non-convex quadrangle $ABCD$ is in the upper half-plane,
$A=(0,0)$ and $B=(1,0)$. \pmn If this pyramid is strongly
flexible, then we are in the scope of the second solution of the
previous theorem. We know that the rotation from the vector
$\overline{AB}$ to the vector $\overline{AC}$ is counter
clockwise, but the rotation from the vector $\overline{AC}$ to the
vector $\overline{AD}$ is clockwise. \pmn As
$$b_1=\dfrac{b_2\cdot(a_2^2-2a_2+b_2^2)}{a^2+b^2}>0,$$ then
$a_2^2-2a_2+b_2^2>0$. The line $BC$ has the equation
$(a_2-1)y-b_2x+b_2=0$. As $b_2>0$ and
$$(a_2-1)\cdot \dfrac{b_2(a_2^2-2a_2+b_2^2)}{a_2^2+b_2^2}-
b_2\cdot\dfrac{a_2^3-a_2^2+a_2b_2^2+b_2^2}{a_2^2+b_2^2}+b_2=
-\dfrac{b_2\cdot(a_2^2-2a_2+b_2^2)}{a_2^2+b_2^2}<0,$$ then
segments $AD$ and $BC$ intersect. \end{proof} \begin{ex} A
self-intersecting quadrangular pyramid can be strongly flexible.
Here is an example. \pmn Let us consider the self-intersecting
pyramid $ABCDE$: $A=(0,0)$, $B=(1,0)$, $C=(2,2)$, $D=(2,1)$,
$E=(1,1,1)$.
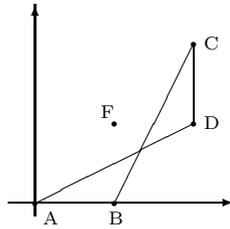
\[\begin{picture}(85,85) \put(0,10){\vector(1,0){85}}
\put(10,5){\vector(0,1){80}} \put(40,10){\line(1,2){30}}
\put(10,10){\line(2,1){60}} \put(70,40){\line(0,1){30}}
\multiput(10,10)(30,0){2}{\circle*{2}}
\multiput(70,40)(0,30){2}{\circle*{2}} \put(40,40){\circle*{2}}
\put(13,2){\scriptsize A} \put(38,2){\scriptsize B}
\put(74,68){\scriptsize C} \put(74,38){\scriptsize D}
\put(35,42){\scriptsize F} \end{picture}\] Here $F$ is the
projection of the apex $E$ to $xy$-plane. We will prove that this
pyramid belongs to a continuous family that realizes strong
flexibility. \pmn Let $A'B'C'D'E'$ be a member of this family:
$A'=(0,0)$, $B'=(1,0)$, $C'=(x_2,y_2)$, $D'=(x_1,y_1)$,
$E'=(x_3,y_3,z_3)$. Then
$$\left\{\begin{array}{l} x_1^2+y_1^2=5\\ (x_2-1)^2+y_2^2=5\\
(x_2-x_1)^2+(y_2-y_1)^2=1\\ x_3^2+y_3^2+z_3^2=3\\
(x_3-1)^2+y_3^2+z_3^2=2\\ (x_3-x_1)^2+(y_3-y_1)^2+z_3^2=2\\
(x_3-x_2)^2+(y_3-y_2)^2+z_3^2=3\end{array}\right. \Rightarrow
\left\{\begin{array}{l} x_1^2+y_1^2=5\\ (x_2-1)^2+y_2^2=5\\
x_1x_2+y_1y_2-x_2=4\\ x_3=1\\ y_3^2+z_3^2=2\\ x_1+y_1y_3=3\\
y_2y_3=2\end{array}\right.$$ Actually equations of this system are
not independent --- all variables are functions of $y_1$:
$$y_1^2y_3^2-6y_1y_3+y_1^2+4=0,\,y_2y_3=2,\,x_1+y_1y3=3,\,
x_1x_2+y_1y_2-x_2=4,\,y_3^2+z_3^2=2.$$ As
$$(y_1^2y_3^2-6y_1y_3+y_1^2+4)'_{y_1}(y_1=1,y_3=1)\neq 0\text{ and }
(y_1^2y_3^2-6y_1y_3+y_1^2+4)'_{y_3}(y_1=1,y_3=1)\neq 0,$$ then we
have continuous family of quadrangular self-intersecting pyramids
whose edges have fixed lengths. \end{ex}

\section{Realizations}
\pn Let lengths of all edges of a labelled quadrangular pyramid
$ABCDE$ are given. As there cannot exist a continuous family of
such pyramids, we can ask about the number of them (pairwise non
congruent). \begin{defin} Let $L$ be a set of eight positive
numbers $L=\{l_1,\ldots,l_8\}$. A \emph{realization} of this set
is a convex quadrangular pyramid $ABCDE$, $ABCD$
--- the base, $E$ --- the apex, such that
$$|AB|=l_1,\,|BC|=l_2,\,|CD|=l_3,\,|DA|=l_4,\,|EA|=l_5,\,|EB|=l_6,\,
|EC|=l_7,\,|ED|=l_8.$$\end{defin} \pn We will assume that $l_1=1$.
\begin{theor} The number of realizations of a set $L$ is $\leqslant 4$.
\end{theor} \begin{proof} Let a convex quadrangular pyramid $ABCDE$ be in the
standard position. Using the notation of the previous section, we
obtain the system
$$\left\{\begin{array}{l} x_1^2+y_1^2=l_4\\ (x_2-1)^2+y_2^2=l_2\\
(x_2-x_1)^2+(y_2-y_1)^2=l_3\\ x_3^2+y_3^2+z_3^2=l_5\\
(x_3-1)^2+y_3^2+z_3^2=l_6\\ (x_3-x_2)^2+(y_3-y_2)^2+z_3^2=l_7\\
(x_3-x_1)^2+(y_3-y_1)^2+z_3^2=l_8\end{array}\right.$$ The elimination of
variables $x_3,y_3,z_3,x_2,y_2,y_1$ gives a polynomial of the forth degree
in $x_1$. \end{proof} \begin{ex} We can give an example of the set $L$,
which has three realizations. \pmn Let $ABCDE$ be a convex quadrangular
pyramid in standard position, where $|BC|=2$, $|CD|=\sqrt 2$, $|DA|=1$,
$|EA|=\sqrt 2$, $|EB|=\sqrt 5$, $|ED|=\sqrt 3$ and the length of the
edge $EC$ we will define later. Using notation of the section 3, we can
write the system
$$\left\{\begin{array}{l} x_1^2+x_2^2=1\\ (x_2-1)^2+y_2^2=4\\
(x_2-x_1)^2+(y_2-y_1)^2=2\\ x_3^2+y_3^2+z_3^2=2\\ (x_3-1)^2+y_3^2+z_3^2=5\\
(x_3-x_1)^2+(y_3-y_1)^2+z_3^2=3\end{array}\right.\Rightarrow
\left\{\begin{array}{l} x_1^2+y_1^2=1\\ x_2^2+y_2^2-2x_2=3\\
x_1x_2+y_1y_2-x_2=1\\ x_3=-1\\ y_3^2+z_3^2=1\\
x_1-y_1y_3=0\end{array} \right.$$ The value of the angle $\angle
A=\alpha$ uniquely defines the quadrangle $ABCD$ and also uniquely
defines the position of the apex $E$. Thus, $|EC|^2$ is the
function of $\alpha$. \pmn the value of $\alpha$ is changed from
the minimal value $\alpha_0\approx 0.9449$ (here points $A$, $C$
and $D$ are on one line and $|EC|^2\approx 7.8284$) to the maximal
value $\alpha_1=3\pi/4$ (here $y_3=-1$, $z_3=0$ and $|EC|^2\approx
9.3067$). \pmn $|EC|^2$ increases on the interval
$(\alpha_0,\pi/2)$. The point $\pi/2$ is the local maximum:
$|EC|^2=9$. Then $|EC|^2$ decreases on the interval
$(\pi/2,\approx 1.9404)$ and in the end of this interval it has
the local minimum $\approx 8.9555$. After that $|EC|^2$ increases
on the interval $(\approx 1.9404,3\pi/4)$. It means that the set
$L=\{1,2,\sqrt 2,1,\sqrt 2,\sqrt 5,r,\sqrt 3\}$, where
$8.9555<r<9$, has three pairwise non-congruent realizations.
\end{ex}

\vspace{5mm}


\begin{thebibliography}{99}

\bibitem{Con1} R. Connely, \emph{A counterexample to the rigidity
conjecture for polyhedra}, Publications Mathematiques de l'IHES, 1977,
47, 333-338.

\bibitem{Con2} R. Connely, \emph{Rigidity}, in "Handbook of Convex
Geometry", North-Holland, 1993, 223-271.

\bibitem{Cox} D. Cox, J. Little, D. O'Shea, \emph{Ideals, Varieties,
and Algorithms}, Springer, 1998.

\bibitem{Lyu} L.A. Lyusternik, \emph{Convex Figures and Polyhedra},
N.Y., Dover, 1963.

\end{thebibliography}
\end{document}